\documentclass[12pt]{article}

\usepackage{amsmath,amsthm,amssymb,amsfonts}
\usepackage[margin=1in]{geometry} 
\usepackage{stmaryrd} 
\usepackage{tikz}
\usetikzlibrary{positioning}

\numberwithin{equation}{section}

\theoremstyle{plain}
\newtheorem{thm}{Theorem}[section]
\newtheorem{lem}[thm]{Lemma}

\newtheorem{prp}[thm]{Proposition}

\theoremstyle{definition}
\newtheorem{dfn}[thm]{Definition}

\theoremstyle{remark}
\newtheorem{rem}[thm]{Remark}

\newtheorem{fct}[thm]{Fact}

\newcommand{\Exp}{\text{Exp}}       

\newcommand{\ZZZ}{\mathbb{Z}}
\newcommand{\NNN}{\mathbb{N}}

\begin{document}

\title{Power graphs of Moufang loops}

\author{Riley Britten}

\maketitle

\section{Introduction}

Power graphs of both groups and semigroups have been widely studied, for example in \cite{EPG}, \cite{PG}, \cite{PGII}, \cite{Tor}, \cite{Feng}, \cite{Mog}, \cite{Panda}. While the power graph of a quasigroup can be defined analogously to that of a group, power graphs of quasigroups and loops have thus far been little studied. In this paper we begin transferring results on the power graphs of groups to the context of loops by addressing a question posed by Peter Cameron: if two Moufang loops have isomorphic undirected power graphs, must they have isomorphic directed power graphs? In \cite{PGII} Cameron shows that two groups with isomorphic undirected power graphs must have isomorphic directed power graphs. We are able to extend that result to Moufang loops in our main theorem:

\begin{thm}\label{mainThm}
	Moufang loops with isomorphic undirected power graphs have isomorphic directed power graphs.
\end{thm}

Cameron's proof in \cite{PGII} relied on handling groups with multiple vertices connected to all others in the power graph separately. Groups with such power graphs are either cyclic or generalized quaternion. We take a similar approach here. In generalizing to Moufang loops, a third type of loop with such a power graph arises; we have termed these \textit{generalized octonion loops}. In {\S}3 we will investigate Moufang loops with power graphs having multiple vertices connected to all others. This will yield the following generalization of a result from group theory:

\begin{thm}\label{moufangDesc2}
	A Moufang $p$-loop $M$ with a unique subloop of order $p$ is either a cyclic group, a generalized quaternion group, or a generalized octonion loop. These last two only occur when $p = 2$.
\end{thm}

In {\S}4 we will investigate the structure of generalized octonion loops, yielding the following characterization:

\begin{thm}\label{genOct}
  A finite Moufang loop is generalized octonion if and only if it is a non-associative Moufang loop
    having order a power of $p$ for some prime $p$ and a unique element of order $p$.
\end{thm}

In {\S}5 we will use the above results to prove Theorem \ref{mainThm}. In {\S}6 we will present miscellaneous results on the power graphs of Moufang loops that arose during our attempts to answer the motivating question.


\section{Preliminaries}

\subsection{Moufang loops}

\begin{dfn}
  A \emph{quasigroup} $(Q, \cdot)$ is a magma whose multiplication table is a Latin square.
\end{dfn}

\begin{dfn}
  A \emph{loop} $(Q, \cdot)$ is a quasigroup with an identity element $1\in Q$ such that
  \[x \cdot 1 = 1\cdot x = x\]
  for all $x\in Q$.
\end{dfn}

Basic references for loop theory are \cite{Bel}, \cite{Bruck}, \cite{Pf}. We will use the convention that juxtaposition binds more tightly than $\cdot$. Any uncited facts in the discussion that follows can be found in these references.

\begin{dfn}
  A loop $(Q, \cdot)$ has the \emph{left inverse property} (LIP) if there exists a bijection $\lambda: Q\to Q$ such that for all $x, y\in Q$
  \[\lambda(x)\cdot xy = y\]
\end{dfn}

\begin{dfn}
  Similarly, a loop $(Q, \cdot)$ has the \emph{right inverse property} (RIP) if there exists a bijection $\gamma: Q\to Q$ such that for all $x, y\in Q$
  \[xy\cdot \gamma(y) = x\]
\end{dfn}

\begin{dfn}
  A loop $(Q, \cdot)$ which has both the left and right inverse properties is said to be an \emph{inverse property loop} (IP loop).
\end{dfn}

\begin{dfn}
  A loop $(Q, \cdot)$ is \emph{power-associative} if $\langle x\rangle$ is a group for all $x\in Q$.
\end{dfn}

\begin{dfn}
  A loop $(Q, \cdot)$ is \emph{diassociative} if $\langle x, y\rangle$ is a group for all $x, y\in Q$.
\end{dfn}

\begin{dfn}
    The \textit{exponent} of a power-associative loop $L$, denoted $\Exp(L)$, is the least common multiple of orders of elements of $L$ if it exists and $0$ otherwise.
\end{dfn}

\begin{dfn}
  A \emph{Moufang loop} is a loop satisfying any (and hence all) of the Moufang identities:
  \begin{align*}
    z\cdot(x \cdot zy) &= (zx\cdot z)\cdot y\\
    x\cdot(z \cdot yz) &= (xz\cdot y)\cdot z\\
    zx\cdot yz &= (z\cdot xy)\cdot z\\
    zx\cdot yz &= z\cdot(xy\cdot z).
  \end{align*}
\end{dfn}

Standard examples of nonassociative Moufang loops are the unit octonions with multiplication and the sphere $S^7$ with octonion multiplication. We now present some fundamental results on Moufang loops which we will need in later sections.

\begin{thm}[Moufang's Theorem]
  Suppose that $M$ is a Moufang loop and $x, y, z\in M$ are such that $x\cdot yz = xy\cdot z$. Then $\langle x, y, z\rangle$ is a group.
\end{thm}

\begin{prp}\label{gen-facts}
	Let $M$ be a finite Moufang loop. Then
    \begin{itemize}
        \item $M$ has the inverse property.

        \item $M$ is diassociative (and thus power-associative).

        \item For all $x, y\in M$, $\langle x, y\rangle$ is a group \cite{Moufang}.

  	    \item For all $x\in M$, $|x|$ divides $|M|$.

        \item Suppose that $|M| = p^k$ for $p$ prime, $k\in\ZZZ^+$. Then there exists $S\leq M$ with $|S| = p^{k - 1}$.
    \end{itemize}
\end{prp}

Regarding the last statement, note that the center of a Moufang $p$-loop is nontrivial \cite{2-loops} \cite{Glau}. So an inductive argument identical to that used to prove the last result for groups will also prove the existence of such a subloop.

\subsection{Generalized quaternion groups}

In the process of proving Theorem \ref{mainThm} we will investigate a class of loops which behave analogously
  to the generalized quaternion groups. We recall some results on generalized quaternion groups to illustrate
  this similarity here.

\begin{dfn}
  The generalized quaternion groups are given by the presentation:
  \[Q_{4n} = \langle a, b : a^n = b^2, a^{2n} = 1, b^{-1}ab = a^{-1}\rangle\]
\end{dfn}

\begin{fct}\label{fct:genQuat}
  Let $G$ be a group which is not cyclic. Then the following are equivalent:
  \begin{itemize}
    \item $G$ is generalized quaternion.
    \item $G$ is isomorphic to $\langle e^{\frac{i\pi}{n}}, j\rangle$ as a subgroup of the unit
      quaternions for some $n$ \cite{Brown}.
    \item $G$ is a finite $p$ group in which every subgroup is cyclic \cite{Cartan}.
    \item $G$ is a finite $p$ group with a unique subloop of order $p$ \cite{Brown}.
  \end{itemize}
\end{fct}

\begin{rem}
  A direct result of fact \ref{fct:genQuat} is: \emph{A finite $p$-group with a unique subloop of order $p$ is
    either cyclic or generalized quaternion}. We will show that this result extends very naturally to Moufang loops.
\end{rem}

\begin{fct}\label{fct:genQuat-rep}
  Let $Q_{4n} = \langle a, b : a^n = b^2, a^{2n} = 1, b^{-1}ab = a^{-1}\rangle$. Then every element $x \in Q_{4n}$
    can be written uniquely as $x = a^k$ or $x = a^k b$ for some $k\in\ZZZ_{2n}$ \cite{genQuat}.
\end{fct}

\subsection{Power graphs}

To maintain generality, in what follows let $\textbf{A} = (A, \cdot)$ be a magma with $\cdot$ a power-associative
  binary operation.

\begin{dfn}
	The \textit{directed power graph} of \textbf{A} is the directed graph with vertex set $A$ and an edge $x\to y$
    if and only if $x^k = y$ for some $k\in\mathbb{Z}$.
\end{dfn}

\begin{dfn}
	The \textit{undirected power graph} of \textbf{A} is the graph with vertex set $A$ and an edge between $x$
    and $y$ if and only if $x^k = y$ for some $k\in \mathbb{Z}$ or $y^k = x$ for some $k\in\mathbb{Z}$.
\end{dfn}

So the undirected power graph of \textbf{A} is the underlying undirected graph of the directed power graph of
  \textbf{A}. In the remainder of this paper, \textit{power graph} will refer to the undirected power graph
  unless otherwise specified.

We will use the following definition for generalized octonion loops in the interest of closely
  following this characterization of generalized quaternion groups. In {\S}4 we will see that there are several
  alternate characterizations of generalized octonion loops.

\begin{dfn}
	Let $M$ be a nonassociative Moufang $p$-loop such that every abelian subloop of $M$ is cyclic,
    then we call $M$ a \textit{generalized octonion loop}.
\end{dfn}

\subsection{Chein's construction}

\begin{thm}[\cite{Chein}]\label{cnst-chein}
  Let $G$ be a group. For $1\neq c\in Z(G)$ and $u$ an indeterminate. Define $(M, \cdot)$ by $M = G\cup Gu$ and
  \begin{align*}
    g\cdot h &= gh\\
    g\cdot (hu) &= (hg)u\\
    gu\cdot h &= (gh^{-1})u\\
    gu\cdot hu &= ch^{-1}g
  \end{align*}
  for all $g, h\in G$. Then $M$ is a Moufang loop \cite{Chein}. Further, $M$ is associative if and only if $G$
    is abelian.
\end{thm}

Throughout the paper we will denote loops arising from this construction by $M(G, 2)$, where $G$ is the underlying
  group. We will show that the loops $M(Q_{4n}, 2)$, where $Q_{4n}$ is
  a generalized quaternion group, are generalized octonion.

\begin{thm}[\cite{Chein}]\label{thm-chein}
  Suppose that $M$ is a finite Moufang loop with a set of generators $\{u, u_1, \ldots, u_n\}$ such that
  \begin{itemize}
    \item $u\notin G = \langle u_1,\ldots, u_n\rangle$,
    \item $u^2\in N(\langle u^2, G\rangle)$,
    \item conjugation by $u$ maps $G$ into itself.
  \end{itemize}
  Let $k$ be the smallest positive integer such that $u^k\in G$. Then
  \begin{itemize}
    \item each element of $M$ can be expressed uniquely as $gu^\alpha$ where $g\in G$ and $0\leq \alpha < k$; and
    \item multiplication of elements of $M$ is given by
    \[(g_1u^\alpha)(g_2u^\beta) = [\theta^{-\beta}(\theta^\beta(g_1)\theta^{\beta - \alpha}(g_2))g_0^\epsilon]u^\rho\]
  \end{itemize}
  where
  \[\theta(g) = u^{-1}gu,\: g_0 = u^k\in G,\: \epsilon = \lfloor\frac{\alpha + \beta}{k}\rfloor,\text{ and }
    \rho = \alpha + \beta - \epsilon k\] \cite{Chein}.
\end{thm}


\section{Moufang $p$-loops with a unique subloop of order $p$}

We will begin by classifying Moufang $p$-loops $M$ with a unique subloop of order $p$. In the proof of
  Theorem \ref{mainThm}, we will handle such loops separately. Note that every nontrivial subloop of a
  Moufang loop of order $p^n$ with a unique subloop of order $p$ also has a unique subloop of order $p$
  by the last point of Proposition \ref{gen-facts}.

\begin{thm}
    A Moufang $p$-loop $M$ with a unique subloop of order $p$ is either a cyclic group, a generalized
      quaternion group, or a generalized octonion loop. These last two only occur when $p = 2$.
\end{thm}

We will first handle the simpler case that $p$ is an odd prime.

\begin{lem}[\cite{Burnside}]\label{lem-burnside}
  Let $G$ be a group of order $p^n$, $p > 2$ prime with a unique subgroup of order $p^s$ for some $0 < s < n$.
    Then $G$ is cyclic. 
\end{lem}

\begin{lem}\label{not2power}
	Let $M$ be a Moufang loop of order $p^n$ for some prime $p > 2$ and $n\in\mathbb{N}$ such that $M$ has a
    unique subloop of order $p$. Then $M$ is a cyclic group.
\end{lem}

\begin{proof}
    Let $x, y, z\in M$ be given. If $\langle x, y\rangle = M$, then $M$ is a group by diassociativity and
      we are done by Lemma \ref{lem-burnside}. Otherwise $\langle x, y\rangle \subsetneq M$
      must be a $p$-group with a unique subgroup of order $p$ and thus cyclic by Lemma \ref{lem-burnside}.
      Say $\langle x, y\rangle = \langle g\rangle$ and $x = g^i$, $y = g^j$. Then
      $x\cdot yz = g^i \cdot g^j z = g^{i + j}z = xy\cdot z$ by diassociativity. Hence in either case 
      $M$ is a group and thus cyclic by Lemma \ref{lem-burnside}.
\end{proof}

We will now handle the case $p = 2$. In what follows, let $M$ be a nonassociative Moufang loop of order $2^n$
  with a unique subloop of order $2$.

\begin{lem}\label{order-lem}
	For all $x, y\in M$ exactly one of the following holds:
  \begin{itemize}
    \item $xy = yx$,
    \item $xy = y^{-1}x$ and $|x| = 4$,
    \item $xy = yx^{-1}$ and $|y| = 4$,
    \item $|x| = |y| = 4$.
  \end{itemize}
\end{lem}

\begin{proof}
	If $\langle x, y\rangle$ is cyclic, then $xy = yx$, so assume that
    $G = \langle x, y\rangle = \langle a, b | a^{2n} = b^4 = 1, ab = ba^{-1}\rangle$ is generalized
    quaternion. All elements of $G$ can be written in the form $a^i b$ or $a^i$ for some
    $i\in\mathbb{N}$. If $x = a^i$, $y = a^j$, then $xy = yx$. If $x = a^i b$,
    then $x^2 = a^i b a^i b = b a^{-i} a^i b = b^2$. Similarly, if $y = a^j b$, then $|y| = 4$.  
    If $x = a^i$, $y = a^j b$, then $xy = a^i a^j b = a^j b a^{-i} = yx^{-1}$.
    Finally, if $x = a^i b$, $y = a^j$, then $xy = a^i b a^j = a^{-j} a^i b = y^{-1} x$.
\end{proof}


\section{Generalized octonion loops}

To make the Theorem \ref{moufangDesc2} more closely follow the result for groups we will investigate the
  the generalized octonion loops. We will show that they behave analogously to generalized quaternion groups.

\begin{thm}
  A finite Moufang loop is generalized octonion if and only if it is a non-associative Moufang $p$-loop
    with a unique element of order $p$.
\end{thm}

\begin{proof}
  First let $M$ be a finite non-associative Moufang $p$-loop with a unique subloop of order $p$. Let $S\leq M$ be
    an associative commutative subloop. Then $S$ has a unique subloop of order $p$ and thus is cyclic by
    Lemma \ref{lem-burnside}. Thus $M$ is generalized octonion.

  Now let $M$ be a generalized quaternion group. By Lemma \ref{not2power} $M$ must be a Moufang $2$-loop. It is
    immediate that $M$ has a subloop of order $2$ by the elementwise Lagrange property. We need only show that it
    is unique. Suppose that $S, T\leq M$ with $|S| = |T| = 2$ and $S\neq T$. Say that $1\neq s\in S$ and $1\neq t\in T$.
    Then $\langle s, t\rangle$ is a group in which every commutative subgroup is cyclic and thus is either cyclic or
    generalized quaternion. But both cyclic $2$-groups and generalized quaternion groups have unique elements of order
    $2$. Thus $s = t$ and $M$ is generalized octonion.
\end{proof}

\begin{thm}\label{doubling}
  $M(Q_{4n}, 2)$ is a generalized octonion loop.
\end{thm}

\begin{proof}
  It is shown in \cite{Chein} that $M$ is a nonassociative Moufang loop. So every associative subloop of
    $M$ is either cyclic or generalized quaternion and thus every commutative and associative subloop of
    $M$ is cyclic. Thus $M$ is generalized octonion.
\end{proof}

Let $\{1, e_1, \ldots, e_7\}$ be the standard basis for the octonions.

\begin{thm}
  The subloop of the unit octonions generated by $\{e^{\frac{e_1\pi}{n}}, e_3, e_5\}$ for some $n\in\NNN$.
    Is generalized octonion.
\end{thm}

\begin{proof}
  Let $M = \langle e^{\frac{e_1\pi}{n}}, e_3, e_5\rangle$ and note that $M$ is nonassociative. We will use
    Theorem 1 in \cite{Chein} to show that this is precisely $M(Q_{4n}, 2)$, taking the presentation
    $Q_{4n}  = \langle e^{\frac{e_1\pi}{n}}, e_3\rangle$. First note that $e_5\notin Q_{4n}$. Further,
    $e_5^2 = -1\in N(\langle -1, Q_{4n}\rangle)$. Finally
	\[e_5 e_3 e_5^{-1} = (e^{\frac{e_1\pi}{n}})^{-1}\in Q_{4n}\]

  Thus $Q_{4n}$ is closed under conjugation by $e_5$ and by Theorem \ref{thm-chein}, $M$ is precisely
    $M(Q_{4n}, 2)$. We showed in Theorem \ref{doubling} that this is the same as $M$ being a generalized
    octonion loop.
\end{proof}

\begin{lem}\label{u-order-4}
  Suppose that $M$ is generalized octonion with $S_1\unlhd S_2\unlhd\ldots\unlhd S_m = M$ where $S_{i + 1}$
    satisfies Theorem \ref{thm-chein} with $G = S_i$. Then there exists such a sequence with $|u| = 4$ in
    Theorem \ref{thm-chein} at each stage.
\end{lem}

\begin{proof}
  Suppose that $S_k$ is the first index at which $|u|\neq 4$. Let $s_0$ be the $u$ in Theorem \ref{thm-chein}
    at this stage. Then $\langle S_1, s_0\rangle$ is a generalized quaternion group since $s_0$ commutes
    with all elements of $S_1$ by Theorem \ref{gen-facts}.

  Now let $1 < i < k$ and $s_i$ be the $u$ in Theorem \ref{thm-chein} for $S_i$. Then $|s_i| = 4$ by assumption
    so $s_i^2$ is the unique element of order $2$ in $M$ and $s_i^2\in N(M)$. Further, $s_i s_0 = s_0 s_i$ since
    $|s_0| > 4$. Thus conjugation by $s_i$ maps $\langle S_{i - 1}, s_0\rangle$ into itself. So the hypotheses
    of Theorem \ref{thm-chein} are satisfied. So we have constructed a sequence
    $\langle S_1, S_0\rangle\unlhd \langle S_2, s_0\rangle\unlhd \ldots \unlhd \langle S_{k - 1}, s_0\rangle = S_k$
    as needed. The same procedure can be repeated for any other stage at which $|u| \neq 4$. So the proof is complete.
\end{proof}

\begin{thm}\label{genOctChain}
  Let $M$ be a generalized octonion loop. Then there exist $S_1\unlhd S_2\unlhd\ldots\unlhd M$ such that
  \begin{enumerate}
    \item $S_{i + 1}$ satisfies the hypotheses of Theorem \ref{thm-chein} with $G = S_i$.
    \item $S_1$ is a generalized quaternion group.
  \end{enumerate}
\end{thm}

\begin{proof}
  Suppose toward a contradiction that $M$ is a minimal counterexample and $|M| = 2^n$. By Theorem
    \ref{gen-facts} there exists $S \leq M$ with $|S| = 2^{n - 1}$. By the minimality of $M$ there exists
    $S_1 \unlhd S_2\unlhd \ldots \unlhd S_k = S$ where $S_{i + 1}$ satisfies Theorem \ref{thm-chein} with
    $G = S_i$ and $S_1$ is generalized quaternion. By Lemma \ref{u-order-4} 2e can without loss of generality
    assume that at each stage the $u$ in Theorem \ref{thm-chein} has order $4$. 

  Let $v\in M - S$ be given. Suppose first that $vs = sv$ for all $s\in S$. Let $\langle s_0\rangle$ be a
    cyclic group of maximal order contained in $S_1$. Then $\langle v, s_0\rangle$ is a cyclic group of
    strictly larger order. We will show that $\langle S_{i + 1}, v\rangle$ satisfies Theorem \ref{thm-chein}
    with $G = \langle S_i, v\rangle$ for all $1\leq i < k$.

  Let $u$ be as in Theorem \ref{thm-chein} for $S_{i + 1}$. Then $u^2$ is the unique element of order $2$ in
    $S_{i + 1}$ and thus is in $N(M)$ and so $N(\langle S_i, v^2\rangle)$. Further conjugation by $u$ maps
    $S_i$ into itself and $uv = vu$, so conjugation by $u$ maps $\langle S_i, v\rangle$ into itself. Thus
    the hypotheses of Theorem \ref{thm-chein} are satisfied.

  Now suppose that $us\neq su$ for some $s\in S$. Then $|u| = 4$ by Lemma \ref{order-lem}. Then $S$ is an index $2$ subloop
    and thus normal in $M$, so conjugation by $u$ takes $S$ into itself. Further, $u^2$ is the unique element of order $2$
    in $S$ since $|u| = 4$. Thus $u^2\in N(\langle u^2, S\rangle)$ and the hypotheses of Theorem \ref{thm-chein} are
    satisfied.
\end{proof}

Recall that the generalized quaternion group of order $4n$ can be presented as
  $Q_{4n} = \langle a, b | a^n = b^2, a^{2n} = 1, b^{-1}ab = a^{-1}\rangle$. Viewing the generalized
  octonion loop of order $16$ as $O_{16} = M(Q_8, 2)$ with this presentation yields the power graph of $O_{16}$
  presented above. Note that the non-identity vertex $a^2$ is connected to all other vertices. We will 
  show later that generalized octonion loops are the only nonassociative Moufang loops with this feature.

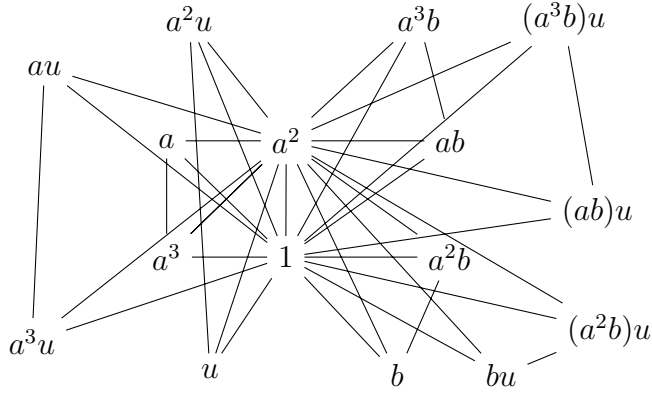
\begin{figure}[h]
  \centering
  \begin{tikzpicture}
  \node (a) at (0, 0){$a$};
  \node (a3) [below = 1cm of a]{$a^3$};
  \node (a2) [right = 1cm of a]{$a^2$};
  \node (1)  [right = 1cm of a3]{$1$};
  \node (a3b) [above right = 1cm and 1cm of a2]{$a^3b$};
  \node (ab) [right = 1.5cm of a2]{$ab$};
  \node (a2b)[right = 1.5cm of 1]{$a^2b$};
  \node (b) [below right = 1cm and 1cm of 1]{$b$};
  
  \node (u) [below left = 1cm and .5cm of 1]{$u$};
  \node (au)[above left = .5cm and 1cm of a]{$au$};
  \node (a3u) [below left = .5cm and 1cm of a3]{$a^3 u$};
  \node (a2u) [above left = 1cm and .5cm of a2]{$a^2u$};
  \node (a3bu) [right = .8cm of a3b]{$(a^3b)u$};
  \node (abu) [below right = .3cm and 1cm of ab]{$(ab)u$};
  \node (a2bu) [below right = .3cm and 1cm of a2b]{$(a^2b)u$};
  \node (bu) [right = .8cm of b]{$bu$};
  
  \draw 	(a) -- (a3)
  (a) -- (1)
  (a) -- (a2)
  (a3) -- (a2)
  (a3) -- (a2)
  (a3) -- (1)
  (a2) -- (a3b)
  (a2) -- (a2b)
  (a2) -- (ab)
  (a2) -- (b)
  (a2) -- (1)
  (a3b) -- (1)
  (a3b) -- (ab)
  (ab) -- (1)
  (a2b) -- (1)
  (a2b) -- (b)
  (b) -- (1)
  
  (u) -- (a2)
  (u) -- (1)
  
  (a3u) -- (a2)
  (a3u) -- (1)
  (a3u) -- (au)
  
  (au) -- (a2)
  (au) -- (1)
  
  (a2u) -- (a2)
  (a2u) -- (1)
  (a2u) -- (u)
  
  (a3bu) -- (a2)
  (a3bu) -- (1)
  (a3bu) -- (abu)
  
  (abu) -- (a2)
  (abu) -- (1)
  
  (a2bu) -- (a2)
  (a2bu) -- (1)
  (a2bu) -- (bu)
  
  (bu) -- (a2)
  (bu) -- (1);
  \end{tikzpicture}
  \caption{The (undirected) power graph of $O_{16}$}
\end{figure}


\section{Undirected power graphs determine directed power graphs} \label{gen-oct}

With Theorem \ref{genOctChain} at our disposal we can now translate the argument in \cite{PGII} to the
  Moufang case to show that two Moufang loops with isomorphic undirected power graphs must have isomorphic
  directed power graphs. As in \cite{PGII} the proof is split into two cases depending on whether the
  identity vertex is the only one connected to all other vertices. In what follows, let $M$ be a Moufang
  loop with power graph $\Gamma$.

\subsection{Non-identity vertex connected to all others}\label{multi-vert}

\begin{lem}\label{connectedVertex}
  Suppose that $x\in M$ with $x\neq 1$ and $x$ is connected to all other vertices in $\Gamma$ and $p$ is a
    prime divisor of $\Exp(M)$. Then $M$ has a unique subgroup $P$ of order $p$ and
    $P = \langle x^n \rangle$ for some $n$.
\end{lem}

\begin{proof}
  Let $|x^p| = k$ and $y\in M$ such that $|y| = p$ be given. Such a $y$ must exist because
    $p|\exp(M)$. Since $x$ and $y$ are connected in
    $\Gamma$, either $x$ is a power of $y$ or $y$ is a power of $x$. Suppose that $y^i = x$.
    Then $(p, i) = 1$ and there exists
    $j\in\mathbb{N}$ such that $x^j = (y^i)^j = y$. Thus every element of $Q$ of order $p$
    is a power of $x$.

  Further $1 = y^p = x^{jp} = (x^p)^j$ and so $k\mid j$. Thus $y = x^j = (x^k)^m$ for some
    $m\in \mathbb{N}$. So every element of order $p$ is contained in $\langle x^k\rangle$,
    a cyclic subgroup of order $p$.
\end{proof}

Thus if the power graph of a Moufang $p$-loop $M$ has a non-identity vertex connected to all
  others, then $M$ is either cyclic, generalized quaternion or generalized octonion by Theorem
  \ref{moufangDesc2}. We now handle the case that $M$ does not have prime power order.

\begin{lem}
  Suppose that $x\in M$ with $x\neq 1$ and $x$ is connected to all other vertices in the power
    graph of $M$ and $|M|$ is not a prime power. Then $M$ is a cyclic group.
\end{lem}

\begin{proof}
  Since $M$ is not a $p$-loop, $\Exp(M)$ is not a prime power \cite{64and81}. As in the proof of
    Lemma \ref{connectedVertex}, $|x|$ is divisible by every prime divisor of $\Exp(M)$. Since
    $\Exp(M)$ has at least two distinct prime factors so does $|x|$. 

  First let $z\in M$ such that $|z| = p^n$ for some prime $p$ be given. Note that either $x$ is
    a power of $z$ or vice versa. But all powers of $z$ have order a power of $p$ while $x$ has
    composite order. So $z$ must be a power of $x$, otherwise we would contradict that the order
    of $x$ is composite.

  Let $z\in M$ such that $|z|$ is not a prime power be given. If $z$ is a power of $x$ then
    we are done. We will show that $z$ must be a power of $x$.

  Suppose toward a contradiction that $z$ is not a power of $x$. Then $x = z^k$ for some $k$.
    Say $|z| = p_0^{i_0}\cdots p_m^{i_m}$, where this is a factorization into distinct primes
    and $m\geq 1$. Then $z^{p_1^{i_1}\cdots p_m^{i_m}}, \ldots, z^{p_0^{i_0}\cdots p_{m - 1}^{i_{m - 1}}}$
    are all powers of $x$ as elements of $M$ with prime power order.
    Thus $p_0^{i_0}, \ldots, p_m^{i_m} | |x|$ and $|z| | |x|$. 
    Then $|x| = |z^k| = \frac{|z|}{\gcd(k, |z|)}$ and $|x| = |z|$. But then
    $|\langle x\rangle| = |\langle z\rangle|$ while $\langle x\rangle \subsetneq \langle z\rangle$,
    a contradiction since $|x|$ is finite.

  Hence every element of $M$ is a power of $x$ and $M = \langle x\rangle$ is a cyclic group.
\end{proof}

We will now prove Theorem \ref{mainThm} in the case that a non-identity vertex is connected to all
  others in $\Gamma$.

\begin{proof}
  From Lemma \ref{connectedVertex} $M$ has a unique subloop of order $p$ for each prime divisor $p$
   of $|M|$. Thus by Theorem \ref{genOct} we have that $M$ is either a cyclic group, a generalized quaternion
    group or a generalized octonion loop. If $|M|$ is not a power of $2$, then $M$ is a cyclic group and there
    is nothing to prove. So suppose that $|M| = 2^n$. Let $K$ be the largest complete subgraph in $\Gamma$.
    We will split the proof into cases based on the size of $K$.

  First suppose that $|K| = 2^n$. Then $M$ is a cyclic group and thus its directed power graph is
    uniquely determined.

  Now suppose that $|K| = 2^{n - 1}$. Then $M$ is a generalized quaternion group and thus its power graph i
    uniquely determined.

  Finally, suppose that $|K| < 2^{n - 1}$. Then $M$ is generalized octonion and there exist
    $S_1\unlhd S_2\unlhd \ldots \unlhd S_k = M$ as in Theorem \ref{genOctChain}. Say $S_1 = Q_m$, the
    generalized quaternion group of order $m$. Choose a subgraph, $\Lambda$, of $\Gamma$ isomorphic to that of $Q_m$
    and apply arrows as in the case of $Q_m$. Now let $u\in M$ be a vertex not in $\Lambda$. By Theorem
    \ref{genOctChain} and Lemma \ref{u-order-4} we have that $u^2$ is the unique element of order $2$ and
    $u^3 = u^{-1}$ is distinct from $u$ and not contained in $\Lambda$.

  So each vertex outside $\Lambda$ is connected to the identity, the unique element of order $2$, and one other vertex
    which also lies outside $\Lambda$. Arrows are directed toward the identity and the unique element of order $2$ and
    are bidirectional between elements outside of $\Lambda$. Thus in this case the direction of each arrow is
    uniquely determined.
\end{proof}

The above takes care of the case that a nonidentity vertex is connected to all others in the power graph. As
    noted above Cameron's proof in \cite{PG} handles the case that the identity is the only vertex connected to
    all others in the power graph. Thus the proof of Theorem \ref{mainThm} is complete.


\section{Future directions of research}

The \textit{enhanced power graph} of a group, which lies between the power graph and the commuting graph as a subgraph, was recently defined in \cite{Zahirovic}. They were able to prove a similar result, that two finite groups with isomorphic power graphs must have isomorphic enhanced power graphs. One natural progression of our research would be to attempt to transfer this result to the context of Moufang loops.

\begin{dfn}
  Let $G$ be a graph, then the \emph{line graph of} $G$ has vertex set edges of $G$,
    with two vertices adjacent if and only if the corresponding edges are incident in $G$
\end{dfn}

\begin{dfn}
  A graph is a \emph{line graph} if it is the line graph of some graph
\end{dfn}

\begin{dfn}
  The \emph{proper power graph} is obtained from the power graph by deleting all vertices
    connected to all others
\end{dfn}

It is shown in \cite{LineGraph} that the proper power graphs of generalized quaternion groups are line graphs.
  We conjecture that this result can be extended to generalized octonion loops.


\bibliographystyle{unsrt}
\bibliography{current_version}

\end{document}